\documentclass{amsart}
\usepackage{setspace}
\usepackage{a4}
\usepackage{amsthm}
\usepackage{latexsym}
\usepackage{amsfonts}
\usepackage{graphicx}
\usepackage{textcomp}
\usepackage{cite}
\usepackage{enumerate}
\usepackage{amssymb}
\usepackage{hyperref}
\usepackage{amsmath}
\usepackage{tikz}
\usepackage[mathscr]{euscript}
\usepackage{mathtools}
\newtheorem{theorem}{Theorem}[section]

\newtheorem{corollary}[theorem] {Corollary}
\newtheorem{definition}[theorem]{Definition}

\newtheorem{proposition}[theorem]{Proposition}

\newtheorem{question}[theorem]{Question}
\title{This is the title}
\usepackage{fancyhdr}

\begin{document}
	\vspace{0.9cm}	
\hrule\hrule\hrule\hrule\hrule
\vspace{0.3cm}	
\begin{center}
{\bf{FUNCTIONAL DONOHO-STARK APPROXIMATE SUPPORT UNCERTAINTY PRINCIPLE}}\\
\vspace{0.3cm}
\hrule\hrule\hrule\hrule\hrule
\vspace{0.3cm}
\textbf{K. MAHESH KRISHNA}\\
Post Doctoral Fellow \\
Statistics and Mathematics Unit\\
Indian Statistical Institute, Bangalore Centre\\
Karnataka 560 059, India\\
Email: kmaheshak@gmail.com\\

Date: \today
\end{center}

\hrule\hrule
\vspace{0.5cm}
\textbf{Abstract}: Let $(\{f_j\}_{j=1}^n, \{\tau_j\}_{j=1}^n)$	and $(\{g_k\}_{k=1}^n, \{\omega_k\}_{k=1}^n)$ be two p-orthonormal bases for a finite dimensional Banach space  $\mathcal{X}$.  If 	$ x \in \mathcal{X}\setminus\{0\}$ is such that $\theta_fx$ is $\varepsilon$-supported on  $M\subseteq \{1,\dots, n\}$ w.r.t. p-norm and $\theta_gx$  is $\delta$-supported on  $N\subseteq \{1,\dots, n\}$ w.r.t. p-norm, then we show that 
\begin{align}\label{ME}
	&o(M)^\frac{1}{p}o(N)^\frac{1}{q}\geq  \frac{1}{\displaystyle \max_{1\leq j,k\leq n}|f_j(\omega_k) |}\max \{1-\varepsilon-\delta, 0\},\\
	&o(M)^\frac{1}{q}o(N)^\frac{1}{p}\geq \frac{1}{\displaystyle \max_{1\leq j,k\leq n}|g_k(\tau_j) |}\max \{1-\varepsilon-\delta, 0\},\label{ME2}
\end{align}
 where 
 \begin{align*}
 	\theta_f: \mathcal{X} \ni x \mapsto (f_j(x) )_{j=1}^n \in \ell^p([n]); \quad \theta_g: \mathcal{X} \ni x \mapsto (g_k(x) )_{k=1}^n \in \ell^p([n])	
 \end{align*}
 and $q$ is the conjugate index of $p$. We call Inequalities (\ref{ME}) and (\ref{ME2}) as \textbf{Functional Donoho-Stark  Approximate Support Uncertainty Principle}. Inequalities (\ref{ME}) and  (\ref{ME2}) improve the finite approximate support uncertainty principle obtained by  Donoho and Stark  \textit{[SIAM J. Appl. Math., 1989]}.



\hrule

\tableofcontents
\hrule
\section{Introduction}
Let $0\leq \varepsilon <1$. Recall that  a function $f \in \mathcal{L}^2	 (\mathbb{R}^d)$   is said to be \textbf{$\varepsilon$-supported on a measurable subset $E\subseteq \mathbb{R}^d$} (also known as \textbf{$\varepsilon$-approximately supported} as well as  \textbf{$\varepsilon$-essentially supported}) \cite{DONOHOSTARK, WILLIAMS} if 	
	\begin{align*}
	\left(\int\limits_{E^c}	|f(x)|^2\,dx \right)^\frac{1}{2}\leq \varepsilon \left(\int\limits_{\mathbb{R}^d}	|f(x)|^2\,dx\right)^\frac{1}{2}.
	\end{align*}
Let $d \in \mathbb{N}$ and  $~\widehat{}:\mathcal{L}^2 (\mathbb{R}^d) \to \mathcal{L}^2 (\mathbb{R}^d)$ be the unitary Fourier transform obtained by extending uniquely the bounded linear operator 
\begin{align*}
	\widehat{}:\mathcal{L}^1 (\mathbb{R}^d)\cap  \mathcal{L}^2	 (\mathbb{R}^d) \ni f \mapsto \widehat{f} \in  C_0(\mathbb{R}^d); \quad \widehat{f}: \mathbb{R}^d \ni \xi \mapsto \widehat{f}(\xi)\coloneqq \int\limits_{\mathbb{R}^d}	f(x)e^{-2\pi i  \langle x, \xi \rangle}\,dx\ \in \mathbb{C}.
\end{align*}
In 1989, Donoho and Stark derived the following uncertainty principle on approximate supports of function and its Fourier transform \cite{DONOHOSTARK}.
\begin{theorem}\cite{DONOHOSTARK}\label{DSA} (\textbf{Donoho-Stark Approximate Support  Uncertainty Principle})  If $f \in \mathcal{L}^2 (\mathbb{R}^d)\setminus \{0\}$ is $\varepsilon$-supported on a measurable subset $E\subseteq \mathbb{R}^d$ and $\widehat{f}$ is $\delta$-supported on a measurable subset $F\subseteq \mathbb{R}^d$, then 
	\begin{align*}
		m(E)m(F)\geq (1-\varepsilon-\delta)^2.
	\end{align*} 
\end{theorem}
Ultimate result in \cite{DONOHOSTARK} is the finite dimensional Heisenberg uncertainty principle known today as Donoho-Stark uncertainty principle. It is then natural to seek a finite dimensional version of Theorem \ref{DSA}. For this, first one needs the notion of approximate support in finite dimensions. Donoho and Stark defined this notion as follows. For $h \in \mathbb{C}^d$, let $\|h\|_0$ be the number of nonzero entries in $h$. Let $\hat{}: \mathbb{C}^d \to  \mathbb{C}^d$ be the Fourier transform. Given a subset  $M\subseteq \{1, \dots, n\}$, the number of elements in $M$ is denoted by $o(M)$. 
\begin{definition}\cite{DONOHOSTARK}\label{DSD}
		Let $0\leq \varepsilon <1$. A vector $(a_j)_{j=1}^d\in  \mathbb{C}^d$ is said to be \textbf{$\varepsilon$-supported on a subset $M\subseteq \{1,\dots, d\}$} if 
	\begin{align*}
		\left(\sum_{j\in M^c}|a_j|^2\right)^\frac{1}{2}\leq \varepsilon 	\left(\sum_{j=1}^d|a_j|^2\right)^\frac{1}{2}.
	\end{align*}
\end{definition}
Finite dimensional version of Theorem \ref{DSA} then reads as follows.
\begin{theorem}\cite{DONOHOSTARK} \label{FDSA} (\textbf{Finite Donoho-Stark Approximate Support  Uncertainty Principle}) 
If 	$ h \in \mathbb{C}^d\setminus\{0\}$ is $\varepsilon$-supported on  $M\subseteq \{1,\dots, d\}$ and $ \widehat{h}$ is $\delta$-supported on  $N\subseteq \{1,\dots, d\}$, then 
\begin{align*}
	o(M)o(N)\geq d (1-\varepsilon-\delta)^2.
\end{align*}
\end{theorem}
In 1990, Smith  \cite{SMITH} generalized Theorem  \ref{FDSA} to Fourier transforms defined on locally compact abelian groups. Recently, Banach space version of finite   Donoho-Stark uncertainty principle has been derived in \cite{KRISHNA3}. Therefore we  seek a Banach space version of Theorem \ref{FDSA}. This we obtain in this paper.

\section{Functional Donoho-Stark Approximate Support Uncertainty Principle}
In the paper,   $\mathbb{K}$ denotes $\mathbb{C}$ or $\mathbb{R}$ and $\mathcal{X}$ denotes a  finite dimensional Banach space over $\mathbb{K}$.  Identity operator on $\mathcal{X}$ is denoted by $I_\mathcal{X}$. Dual of $\mathcal{X}$ is denoted by $\mathcal{X}^*$. Whenever $1<p<\infty$, $q$ denotes the conjugate index of $p$. For $d \in \mathbb{N}$, the standard finite dimensional Banach space $\mathbb{K}^d$ over $\mathbb{K}$ equipped with standard $\|\cdot\|_p$ norm is denoted by $\ell^p([d])$. Canonical basis for $\mathbb{K}^d$ is denoted by $\{e_j\}_{j=1}^d$ and $\{\zeta_j\}_{j=1}^d$ be the coordinate functionals associated with $\{e_j\}_{j=1}^d$. 
\begin{definition}\label{PONB}\cite{KRISHNA}
	Let $\mathcal{X}$  be a  finite dimensional Banach space over $\mathbb{K}$.   Let $\{\tau_j\}_{j=1}^n$ be a basis for   $\mathcal{X}$ and 	let $\{f_j\}_{j=1}^n$ be the coordinate functionals associated with $\{\tau_j\}_{j=1}^n$. The pair $(\{f_j\}_{j=1}^n, \{\tau_j\}_{j=1}^n)$ is said to be a \textbf{p-orthonormal basis} ($1<p <\infty$) for $\mathcal{X}$ if  the following conditions hold.
	\begin{enumerate}[\upshape(i)]
		\item $\|f_j\|=\|\tau_j\|=1$ for all $1\leq j\leq n$.
		\item For every $(a_j)_{j=1}^n \in \mathbb{K}^n$, 
		\begin{align*}
		\left\|\sum_{j=1}^na_j\tau_j \right\|=\left(\sum_{j=1}^n|a_j|^p\right)^\frac{1}{p}.
		\end{align*}
		\end{enumerate}
\end{definition}
Given a p-orthonormal basis $(\{f_j\}_{j=1}^n, \{\tau_j\}_{j=1}^n)$ for $\mathcal{X}$, we get the following two invertible isometries:
\begin{align*}
	\theta_f: \mathcal{X} \ni x \mapsto (f_j(x))_{j=1}^n \in \ell^p([n]), \quad 
	\theta_\tau :\ell^p([n])\ni (a_j)_{j=1}^n \mapsto \sum_{j=1}^{n}a_j \tau_j \in \mathcal{X}.
\end{align*}
Then we have the following proposition.
\begin{proposition}
Let $(\{f_j\}_{j=1}^n, \{\tau_j\}_{j=1}^n)$ be a p-orthonormal basis  for $\mathcal{X}$. Then 
\begin{enumerate}[\upshape(i)]
	\item $\theta_f$ is an invertible isometry.
	\item $\theta_\tau$ is an invertible isometry.
	\item $\theta_\tau\theta_f=I_\mathcal{X}$.
\end{enumerate}
\end{proposition}
It is natural to guess the following version of Definition \ref{DSD} for $ \ell^p([n])$. 
\begin{definition}
	Let $0\leq \varepsilon <1$. A vector $(a_j)_{j=1}^n\in  \ell^p([n])$ is said to be $\varepsilon$-supported on a subset $M\subseteq \{1.\dots, n\}$ w.r.t. p-norm if 
	\begin{align*}
		\left(\sum_{j\in M^c}|a_j|^p\right)^\frac{1}{p}\leq \varepsilon 	\left(\sum_{j=1}^n|a_j|^p\right)^\frac{1}{p}.
	\end{align*}
\end{definition}
With the above definition we have following theorem.
\begin{theorem}\label{ODSA}
(\textbf{Functional Donoho-Stark Approximate Support Uncertainty Principle})
Let $(\{f_j\}_{j=1}^n, \{\tau_j\}_{j=1}^n)$	and $(\{g_k\}_{k=1}^n, \{\omega_k\}_{k=1}^n)$ be two p-orthonormal bases for a finite dimensional Banach space  $\mathcal{X}$.  If 	$ x \in \mathcal{X}\setminus\{0\}$ is such that $\theta_fx$ is $\varepsilon$-supported on  $M\subseteq \{1,\dots, n\}$ w.r.t. p-norm and $\theta_gx$  is $\delta$-supported on  $N\subseteq \{1,\dots, n\}$ w.r.t. p-norm, then 
\begin{align}
		&o(M)^\frac{1}{p}o(N)^\frac{1}{q}\geq  \frac{1}{\displaystyle \max_{1\leq j,k\leq n}|f_j(\omega_k) |}\max \{1-\varepsilon-\delta, 0\},\label{ADS}\\
	&o(M)^\frac{1}{q}o(N)^\frac{1}{p}\geq \frac{1}{\displaystyle \max_{1\leq j,k\leq n}|g_k(\tau_j) |}\max \{1-\varepsilon-\delta,0\}.\label{ADS2}
\end{align}
\end{theorem}
\begin{proof}
	For $S\subseteq \{1, \dots, n\}$, define $P_S: \ell^p([n]) \ni (a_j)_{j=1}^n \mapsto \sum_{j\in S} a_j e _j \in \ell^p([n]) $ be the canonical projection onto the coordinates indexed by $S$. Now define $V\coloneqq P_M\theta_f \theta_\omega P_N: \ell^p([n]) \to  \ell^p([n]) $. Then for $z \in \ell^p([n])$,
	 
	\begin{align*}
		&\|Vz\|^p=	\|P_M\theta_f \theta_\omega P_Nz\|^p=\left\| P_M\theta_f \theta_\omega P_N\left(\sum_{k=1}^{n}\zeta_k(z)e_k\right)\right\|^p=\left\| P_M\theta_f \theta_\omega \left(\sum_{k=1}^{n}\zeta_k(z)P_Ne_k\right)\right\|^p\\
		&=\left\| P_M\theta_f \theta_\omega \left(\sum_{k\in N}\zeta_k(z)e_k\right)\right\|^p=\left\| P_M\theta_f  \left(\sum_{k\in N}\zeta_k(z)\theta_\omega e_k\right)\right\|^p=\left\| P_M\theta_f  \left(\sum_{k\in N}\zeta_k(z)\omega_k\right)\right\|^p\\
		&=\left\|\sum_{k\in N}\zeta_k(z)P_M\theta_f \omega_k\right\|^p=\left\|\sum_{k\in N}\zeta_k(z)P_M\left(\sum_{j=1}^{n}f_j(\omega_k)e_j\right)\right\|^p=\left\|\sum_{k\in N}\zeta_k(z)\sum_{j=1}^{n}f_j(\omega_k)P_Me_j\right\|^p\\
		&=\left\|\sum_{k\in N}\zeta_k(z)\sum_{j\in M}f_j(\omega_k)e_j\right\|^p=\left\|\sum_{j\in M}\left(\sum_{k\in N} \zeta_k(z)f_j(\omega_k)\right)e_j\right\|^p=\sum_{j\in M}\left|\sum_{k\in N} \zeta_k(z)f_j(\omega_k)\right|^p\\
		&\leq \sum_{j\in M}\left(\sum_{k\in N}| \zeta_k(z)f_j(\omega_k)|\right)^p\leq \left(\max_{1\leq j,k\leq n}|f_j(\omega_k) |\right)^p\sum_{j\in M}\left(\sum_{k\in N}| \zeta_k(z)|\right)^p\\
		&=\left(\max_{1\leq j,k\leq n}|f_j(\omega_k) |\right)^po(M)\left(\sum_{k\in N}| \zeta_k(z)|\right)^p\leq \left(\max_{1\leq j,k\leq n}|f_j(\omega_k) |\right)^po(M)\left(\sum_{k\in N}| \zeta_k(z)|^p\right)^\frac{p}{p}\left(\sum_{k\in N}1^q\right)^\frac{p}{q}\\
		&=\left(\max_{1\leq j,k\leq n}|f_j(\omega_k) |\right)^po(M)\left(\sum_{k\in N}| \zeta_k(z)|^p\right)^\frac{p}{p}o(N)^\frac{p}{q}\leq \left(\max_{1\leq j,k\leq n}|f_j(\omega_k) |\right)^po(M)\left(\sum_{k=1}^n| \zeta_k(z)|^p\right)^\frac{p}{p}o(N)^\frac{p}{q}\\
		&=\left(\max_{1\leq j,k\leq n}|f_j(\omega_k) |\right)^po(M)\|z\|^po(N)^\frac{p}{q}.
	\end{align*}
Therefore 
\begin{align}\label{P1}
	\|V\|\leq \left(\max_{1\leq j,k\leq n}|f_j(\omega_k) |\right)o(M)^\frac{1}{p}o(N)^\frac{1}{q}.
\end{align}
We now wish to find a lower bound on the operator norm of $V$. For $x\in \mathcal{X}$, we find  
\begin{align*}
	\|\theta_fx-V\theta_gx\|&\leq \|\theta_fx-P_M\theta_fx\|+\|P_M\theta_fx-V\theta_gx\|\leq \varepsilon \|\theta_fx\|+\|P_M\theta_fx-V\theta_gx\|\\
	&=\varepsilon \|\theta_fx\|+\|P_M\theta_fx-P_M\theta_f \theta_\omega P_N\theta_gx\|=\varepsilon \|\theta_fx\|+\|P_M\theta_f(x-\theta_\omega P_N\theta_gx)\|\\
	&\leq \varepsilon \|\theta_fx\|+\|x-\theta_\omega P_N\theta_gx\|= \varepsilon \|\theta_fx\|+\|\theta_\omega \theta_gx-\theta_\omega P_N\theta_gx\|\\
	&=\varepsilon \|\theta_fx\|+\|\theta_\omega (\theta_gx- P_N\theta_gx)\|=\varepsilon \|\theta_fx\|+\|\theta_gx- P_N\theta_gx\|\\
	&\leq \varepsilon \|\theta_fx\|+\delta\|\theta_gx\|=\varepsilon \|x\|+\delta\|x\|=(\varepsilon +\delta)\|x\|.
\end{align*}
Using triangle inequality, we then get 
\begin{align*}
	\|x\|-\|V\theta_gx\|=	\|\theta_fx\|-\|V\theta_gx\|\leq 	\|\theta_fx-V\theta_gx\|\leq 	(\varepsilon +\delta)\|x\|, \quad \forall x \in \mathcal{X}.
\end{align*}
Since $\theta_g$ is an invertible isometry, 
\begin{align*}
	(1-\varepsilon-\delta) \|x\|\leq \|V\theta_gx\|, \quad \forall x \in \mathcal{X} \implies (1-\varepsilon-\delta) \|y\|=(1-\varepsilon-\delta) \|\theta_g^{-1}y\|\leq \|Vy\|, \quad \forall y  \in \ell^p([n]),
\end{align*}
i.e., 
\begin{align}\label{P2}
	\max \{1-\varepsilon-\delta,0\} \leq \|V\|.
\end{align}
Using Inequalities  (\ref{P1}) and (\ref{P2}) we get 
\begin{align*}
	\max \{1-\varepsilon-\delta,0\} \leq 	\left(\max_{1\leq j,k\leq n}|f_j(\omega_k) |\right)o(M)^\frac{1}{p}o(N)^\frac{1}{q}.
\end{align*}
To prove second inequality, define $W\coloneqq P_N\theta_g \theta_\tau P_M: \ell^p([n]) \to  \ell^p([n]) $. Then for $z \in \ell^p([n])$, 
\begin{align*}
	&\|Wz\|^p=	\|P_N\theta_g \theta_\tau P_Mz\|^p=\left\| P_N\theta_g \theta_\tau P_M\left(\sum_{j=1}^{n}\zeta_j(z)e_j\right)\right\|^p=\left\| P_N\theta_g \theta_\tau \left(\sum_{j=1}^{n}\zeta_j(z)P_Me_j\right)\right\|^p\\
	&=\left\| P_N\theta_g \theta_\tau \left(\sum_{j\in M}\zeta_j(z)e_j\right)\right\|^p=\left\| P_N\theta_g  \left(\sum_{j\in M}\zeta_j(z)\theta_\tau e_j\right)\right\|^p=\left\| P_N\theta_g  \left(\sum_{j\in M}\zeta_j(z)\tau_j\right)\right\|^p\\
	&=\left\|   \sum_{j\in M}\zeta_j(z)P_N\theta_g\tau_j\right\|^p=\left\|   \sum_{j\in M}\zeta_j(z)P_N\left(\sum_{k=1}^ng_k(\tau_j)e_k\right)\right\|^p=\left\|   \sum_{j\in M}\zeta_j(z)\sum_{k=1}^ng_k(\tau_j)P_Ne_k\right\|^p\\
	&=\left\|   \sum_{j\in M}\zeta_j(z)\sum_{k\in N}g_k(\tau_j)e_k\right\|^p=\left\|   \sum_{k\in N}\left(\sum_{j\in M}\zeta_j(z)g_k(\tau_j)\right)e_k\right\|^p= \sum_{k\in N}\left|\sum_{j\in M}\zeta_j(z)g_k(\tau_j)\right|^p\\
	&\leq \sum_{k\in N}\left(\sum_{j\in M}|\zeta_j(z)g_k(\tau_j)|\right)^p\leq 	\left(\max_{1\leq j,k\leq n}|g_k(\tau_j) |\right)^p\sum_{k\in N}\left(\sum_{j\in M}|\zeta_j(z)|\right)^p\\
	&=\left(\max_{1\leq j,k\leq n}|g_k(\tau_j) |\right)^po(N)\left(\sum_{j\in M}|\zeta_j(z)|\right)^p\leq \left(\max_{1\leq j,k\leq n}|g_k(\tau_j) |\right)^po(N)\left(\sum_{j\in M}|\zeta_j(z)|^p\right)^\frac{p}{p}\left(\sum_{j\in M}1^q\right)^\frac{p}{q}\\
	&=\left(\max_{1\leq j,k\leq n}|g_k(\tau_j) |\right)^po(N)\left(\sum_{j\in M}|\zeta_j(z)|^p\right)^\frac{p}{p}o(M)^\frac{p}{q}\leq \left(\max_{1\leq j,k\leq n}|g_k(\tau_j) |\right)^po(N)\left(\sum_{j=1}^n|\zeta_j(z)|^p\right)^\frac{p}{p}o(M)^\frac{p}{q}\\
	&=\left(\max_{1\leq j,k\leq n}|g_k(\tau_j) |\right)^po(N)\|z\|^po(M)^\frac{p}{q}.
\end{align*}
Therefore 
\begin{align}\label{P3}
	\|W\|\leq \left(\max_{1\leq j,k\leq n}|g_k(\tau_j) |\right)o(M)^\frac{1}{q}o(N)^\frac{1}{p}.
\end{align}
Now for $x\in \mathcal{X}$, 
\begin{align*}
	\|\theta_gx-W\theta_fx\|&\leq \|\theta_gx-P_N\theta_gx\|+\|P_N\theta_gx-W\theta_fx\|\leq \delta \|\theta_gx\|+\|P_N\theta_gx-W\theta_fx\|\\
&=\delta \|\theta_gx\|+\|P_N\theta_gx-P_N\theta_g \theta_\tau P_M\theta_fx\|=\delta \|\theta_gx\|+\|P_N\theta_g(x-\theta_\tau P_M\theta_fx)\|\\
&\leq \delta \|\theta_gx\|+\|x-\theta_\tau P_M\theta_fx\|= \delta \|\theta_gx\|+\|\theta_\tau \theta_fx-\theta_\tau P_M\theta_fx\|\\
&=\delta \|\theta_gx\|+\|\theta_\tau (\theta_fx- P_M\theta_fx)\|=\delta \|\theta_gx\|+\|\theta_fx- P_M\theta_fx\|\\
&\leq \delta \|\theta_gx\|+\varepsilon\|\theta_fx\|=\delta\|x\|+\varepsilon \|x\|=(\delta+\varepsilon)\|x\|.
\end{align*}
Using triangle inequality and the fact that $\theta_f$ is an invertible isometry we then get 
\begin{align}\label{P4}
	\max \{1-\varepsilon-\delta,0\} \leq \|W\|.	
\end{align}
Using Inequalities  (\ref{P3}) and (\ref{P4}) we get 
\begin{align*}
	\max \{1-\varepsilon-\delta, 0\} \leq 	\left(\max_{1\leq j,k\leq n}|g_k(\tau_j) |\right)o(M)^\frac{1}{q}o(N)^\frac{1}{p}.
\end{align*}
\end{proof}
\begin{corollary}
	Let $\{\tau_j\}_{j=1}^n$ and   $\{\omega_j\}_{j=1}^n$ be two orthonormal bases   for a  finite dimensional Hilbert space $\mathcal{H}$. Set 
	
	\begin{align*}
		\theta_\tau: \mathcal{H} \ni h \mapsto (\langle h, \tau_j\rangle)_{j=1}^n \in \mathbb{C}^n,\quad 
		\theta_\omega: \mathcal{H} \ni h \mapsto (\langle h, \omega_j\rangle)_{j=1}^n \in \mathbb{C}^n.
	\end{align*}
If 	$ h \in \mathcal{H}\setminus\{0\}$ is such that $\theta_\tau h$ is $\varepsilon$-supported on  $M\subseteq \{1,\dots, n\}$ and $\theta_\omega h$  is $\delta$-supported on  $N\subseteq \{1,\dots, n\}$, then 
	\begin{align*}
		o(M)o(N)\geq \frac{1}{\displaystyle\max_{1\leq j,k\leq n}|\langle \tau_j, \omega_k \rangle |^2} (1-\varepsilon-\delta)^2.
	\end{align*}
	 In particular, 	Theorem \ref{FDSA} follows from Theorem \ref{ODSA}.
\end{corollary}
\begin{proof}
	Define 
	\begin{align*}
		f_j:\mathcal{H} \ni h \mapsto \langle h, \tau_j \rangle \in \mathbb{K}; \quad g_j:\mathcal{H} \ni h \mapsto \langle h, \omega_j \rangle \in \mathbb{K}, \quad \forall 1\leq j\leq n.
	\end{align*}
	Then $p=q=2$ and $	|f_j(\omega_k)|=|\langle \omega_k, \tau_j \rangle | $  for all $1\leq j, k \leq n.$	Theorem  \ref{FDSA} follows by taking $\{\tau_j\}_{j=1}^n$ as the standard basis and $\{\omega_j\}_{j=1}^n$ as the Fourier basis for $\mathbb{C}^n$. 
\end{proof}
\begin{corollary}
Let $(\{f_j\}_{j=1}^n, \{\tau_j\}_{j=1}^n)$	and $(\{g_k\}_{k=1}^n, \{\omega_k\}_{k=1}^n)$ be two p-orthonormal bases for a finite dimensional Banach space  $\mathcal{X}$.  Let 	$ x \in \mathcal{X}\setminus\{0\}$ is such that $\theta_fx$ is $\varepsilon$-supported on  $M\subseteq \{1,\dots, n\}$ w.r.t. p-norm and $\theta_gx$  is $\delta$-supported on  $N\subseteq \{1,\dots, n\}$ w.r.t. p-norm. If $\varepsilon+\delta \leq 1$, then 
\begin{align*}
	&o(M)^\frac{1}{p}o(N)^\frac{1}{q}\geq  \frac{1}{\displaystyle \max_{1\leq j,k\leq n}|f_j(\omega_k) |}(1-\varepsilon-\delta),\\
	&o(M)^\frac{1}{q}o(N)^\frac{1}{p}\geq \frac{1}{\displaystyle \max_{1\leq j,k\leq n}|g_k(\tau_j) |}(1-\varepsilon-\delta).
\end{align*}	
\end{corollary}
\begin{corollary}\label{SC}
Let $(\{f_j\}_{j=1}^n, \{\tau_j\}_{j=1}^n)$	and $(\{g_k\}_{k=1}^n, \{\omega_k\}_{k=1}^n)$ be two p-orthonormal bases for a finite dimensional Banach space  $\mathcal{X}$.  If 	$ x \in \mathcal{X}\setminus\{0\}$ is such that $\theta_fx$ is $0$-supported on  $M\subseteq \{1,\dots, n\}$ w.r.t. p-norm and $\theta_gx$  is $0$-supported  on  $N\subseteq \{1,\dots, n\}$  w.r.t. p-norm (saying differently, $\theta_fx$ is supported on $M$ and $\theta_gx$ is supported on $N$), then 
\begin{align*}
	o(M)^\frac{1}{p}o(N)^\frac{1}{q}\geq  \frac{1}{\displaystyle \max_{1\leq j,k\leq n}|f_j(\omega_k) |},\quad o(M)^\frac{1}{q}o(N)^\frac{1}{p}\geq \frac{1}{\displaystyle \max_{1\leq j,k\leq n}|g_k(\tau_j) |}.
\end{align*}	
\end{corollary}
	Corollary \ref{SC} is not the Theorem 2.3  in \cite{KRISHNA3} (it is a particular case) because Theorem 2.3 in \cite{KRISHNA3} is derived for p-Schauder frames which is general than p-orthonormal bases.
Theorem  \ref{ODSA}  promotes  the following question.
\begin{question}
	Given $p$ and a Banach space $\mathcal{X}$ of dimension $n$, for which pairs of p-orthonormal bases $(\{f_j\}_{j=1}^n, \{\tau_j\}_{j=1}^n)$, $(\{g_k\}_{k=1}^n, \{\omega_k\}_{k=1}^n)$ for $\mathcal{X}$, subsets $M,N$ and $\varepsilon, \delta$, we have equality in Inequalities (\ref{ADS}) and (\ref{ADS2})?
\end{question}
Observe that  we used $1<p<\infty$ in the proof of  Theorem  \ref{ODSA}. Therefore we have the following problem.
\begin{question}
	Whether there are Functional Donoho-Stark Approximate Support  Uncertainty Principle (versions of Theorem  \ref{ODSA}) for 1-orthonormal bases and $\infty$-orthonormal bases?
\end{question}
Keeping $\ell^p$-spaces for $0<p<1$ as a model space equipped with 
\begin{align*}
\|(a_j)_{j=1}^n\|_p\coloneqq \sum_{j=1}^{n}|a_j|^p, \quad \forall (a_j)_{j=1}^n \in \mathbb{K}^n,
\end{align*}
we  set  following definitions.
\begin{definition}
Let $\mathcal{X}$  be a  vector   space over $\mathbb{K}$. We say that 	$\mathcal{X}$ is a \textbf{disc-Banach space} if there exists a map called as \textbf{disc-norm} $\|\cdot\|:\mathcal{X} \to [0, \infty)$ satisfying the following conditions.
\begin{enumerate}[\upshape(i)]
	\item If $x \in \mathcal{X} $ is such that $\|x\|=0$, then $x=0$.
	\item $\|x+y\|\leq \|x\|+\|y\|$ for all $x, y  \in \mathcal{X}$. 
	\item $\|\lambda x\|\leq |\lambda|\|x\|$ for all $x  \in \mathcal{X}$ and for all $\lambda\in\mathbb{K}$ with $|\lambda|\geq 1$.
	\item $\|\lambda x\|\geq |\lambda|\|x\|$ for all $x  \in \mathcal{X}$ and for all $\lambda\in\mathbb{K}$ with $|\lambda|\leq 1$.
	\item $\mathcal{X}$ is complete w.r.t. the metric $d(x, y)\coloneqq \|x-y\|$ for all $x, y  \in \mathcal{X}$. 
\end{enumerate}
\end{definition}
\begin{definition}
Let $\mathcal{X}$  be a  finite dimensional disc-Banach space   over $\mathbb{K}$.   Let $\{\tau_j\}_{j=1}^n$ be a basis for   $\mathcal{X}$ and 	let $\{f_j\}_{j=1}^n$ be the coordinate functionals associated with $\{\tau_j\}_{j=1}^n$. The pair $(\{f_j\}_{j=1}^n, \{\tau_j\}_{j=1}^n)$ is said to be a \textbf{p-orthonormal basis} ($1<p <\infty$) for $\mathcal{X}$ if  the following conditions hold.
\begin{enumerate}[\upshape(i)]
	\item $\|f_j\|=\|\tau_j\|=1$ for all $1\leq j\leq n$.
	\item For every $(a_j)_{j=1}^n \in \mathbb{K}^n$, 
	\begin{align*}
		\left\|\sum_{j=1}^na_j\tau_j \right\|=\sum_{j=1}^n|a_j|^p.
	\end{align*}
\end{enumerate}	
\end{definition}
 Then we also have the following question.
\begin{question}
	Whether there are versions of Theorem  \ref{ODSA} for p-orthonormal bases $0<p<1$?
\end{question}
We wish to mention that in   \cite{KRISHNA3} the functional uncertainty principle was derived for p-Schauder frames which is general than p-orthonormal bases. Thus it is desirable to derive  Theorem   \ref{ODSA} or a variation of it  for p-Schauder frames, which we can't.

We end by asking the following curious question whose motivation is the recently proved Balian-Low theorem (which is also an uncertainty principle) for Gabor systems in finite dimensional Hilbert spaces \cite{NITZANOLSEN2, NITZANOLSEN, LAMMERSSTAMPE}.
\begin{question}\label{FBL}
	\textbf{Whether there is a Functional Balian-Low Theorem (which we like to call Functional Balian-Low-Lammers-Stampe-Nitzan-Olsen Theorem) for Gabor-Schauder  systems in finite dimensional Banach spaces (Gabor-Schauder system is as defined in \cite{KRISHNA4})?}
\end{question}

 \bibliographystyle{plain}
 \bibliography{reference.bib}

\end{document}